\documentclass{hussein}
\title{3-Colouring Graphs Excluding a Fixed Minor}

\author{Vida Dujmović\thanks{School of Computer Science and Electrical Engineering, University of Ottawa.} \and Hussein Houdrouge\thanks{School of Computer Science, Carleton University. Research partially funded by NSERC.} \and Pat Morin\footnotemark[2]}

\date{}

\begin{document}
\maketitle
\begin{abstract}
	We show that, for every fixed graph $H$, every $n$-vertex graph $G$ that excludes $H$ as a minor is $3$-colourable with clustering $O_H(n^{4/9})$. That is, there exists a function $f$ such that for every graph $H$, every $n\ge 1$, every $n$-vertex graph $G$ that excludes $H$ as a minor has a vertex colouring with $3$ colours in which each monochromatic component has size at most $f(H)\cdot n^{4/9}$.  This generalizes a recent result of Dujmović, Morin, Norin, and Wood (\textit{arXiv}:2507.03163) from planar graphs to all proper minor-closed graph classes and is the first improvement on clustered $3$-colouring of proper minor-closed graph classes since the upper bound of $O_H(\sqrt{n})$ due to Linial, Matoušek, Sheffet, and Tardos (\textit{Comb. Prob. Comput.}, \textbf{17}(4):577–589, 2008).
\end{abstract}
\vspace{3mm}
\hrule
\vspace{5mm}
\section{Introduction}
A \defin{vertex colouring} of a graph $G$ is a function $\varphi: V(G) \to \K$ where the elements of the set $\K$ are called \defin{colours}.  A  vertex colouring is \defin{proper} if $\varphi(u) \neq \varphi(v)$ for every edge $uv$ of $G$.  For $k \in \N$, we say a graph $G$ is \defin{$k$-colourable} if there exists a proper colouring $\varphi: V(G) \to \K$ where the cardinality of $\K$ is at most $k$. Proper colouring is a widely studied subject within graph theory. The celebrated \defin{Four Colour Theorem} for planar graphs asserts that every planar graph is $4$-colourable \cite{ROBERTSON19972, appel1989every}.  A major open problem is the \defin{Hadwiger Conjecture}, which asserts that every graph that does not contain the complete graph $K_t$ on $t$ vertices as a minor can be properly coloured with $t-1$ colours.

In the current paper, we study a relaxation of proper colouring.
Given a (not necessarily proper) colouring $\varphi$ of a graph $G$, a \defin{cluster} is a connected subgraph of $G$ whose vertices are all assigned to the same colour. A vertex-maximal cluster is called a \defin{monochromatic component}.  We say that a colouring of a graph $G$ has \defin{clustering $c$} if the size of every cluster is at most $c$.  Thus, a proper colouring is a colouring with clustering one. Recently, clustered colouring has received considerable attention, as a way of approaching difficult conjectures on proper colouring (such as the Hadwiger Conjecture or Hajós Conjecture \cite{dvořák2017islandsminorclosedclassesi,focsHWC, RelHW, ImproperColHW, KAWARABAYASHI_2008, KAWARABAYASHI2007647, LIU202227}), as an interesting question in its own right (see \cite{dujmović20253colouringplanargraphs, briański2024defectiveclusteredcolouringgraphs,CROUCH202528,Dujmović_bounded_degree, Hendrey_Wood_2019, oddhminorfreegraphs, Kang_Oum_2019, LINIAL_MATOUŠEK_SHEFFET_TARDOS_2008, DefColPerfect, LIU2018114, liu2021clusteredcoloringgraphsexcluding, liu2022clusteredgraphcoloringlayered, LIU2024103730, liu2024quasitreepartitionsgraphsexcludedsubgraph, MoharClustered, Norin20191387, Norin_Scott_Wood_2023, Esperet2014551}), and for its application in databases as in \cite{focsKleinbergMRV97}.  For further details, we refer the reader to the survey by \citet{woodservey}.

One line of research considered by \citet{LINIAL_MATOUŠEK_SHEFFET_TARDOS_2008} asks about colouring planar graphs with two colours or three colours, rather than four.  The authors show that planar graphs have $2$-colourings with clustering $O(n^{2/3})$ (which follows quickly from the Planar Separator Theorem \cite{lt-stpg-79}) and show that this is tight: there are $n$-vertex planar graphs in which every $2$-colouring has a cluster of size $\Omega(n^{2/3})$.  The same authors show that $n$-vertex planar graphs (and, more generally, all proper-minor-closed families of graphs) have $3$-colourings with clustering $O(n^{1/2})$ (which again follows quickly from the Planar Separator Theorem \cite{lt-stpg-79} and its extension to proper minor-closed graph classes \cite{AlonSeymourThomas}).  On the lower bound side, there are some $n$-vertex planar graphs in which every $3$-colouring has a cluster of size $\Omega(n^{1/3})$.  These two natural-looking bounds stood for many years, until recently, when \citet{dujmović20253colouringplanargraphs} showed that planar graphs have $3$-colourings with clustering $O(n^{4/9})$:

\begin{theorem}[\citet{dujmović20253colouringplanargraphs}]\label{planar_main_theorem}
  Every $n$-vertex planar graph has a $3$-colouring with clustering $O(n^{4/9})$.
\end{theorem}

In this paper, we consider  a natural follow-up question: can \cref{planar_main_theorem} be generalized to graphs from any proper minor-closed family of graphs?  We answer this question positively by successive generalization of this result. We first show that it can be generalized to all graphs of bounded genus, and then to every graph excluding a fixed minor.

\begin{theorem}\label{thm::main}\label{genus_main_theorem}
	For every fixed $g\in\N$, every $n$-vertex graph of Euler-genus at most $g$ has a $3$-colouring with clustering $O_g(n^{4/9})$.
\end{theorem}

\begin{theorem}\label{thm::3-colouring_H_free_minor}\label{minor_free_main_theorem}
	For every fixed graph $H$, every $n$-vertex $H$-minor-free graph $G$ has a $3$-colouring with clustering $O_H(n^{4/9})$.
\end{theorem}

\Cref{minor_free_main_theorem} is somewhat surprising because both the number of colours ($3$) and the asymptotic cluster size ($O(n^{4/9})$) are independent of the graph $H$.  This contrasts with a line of work on the Clustered Hadwiger Conjecture that shows that, for $K_t$-minor-free graphs, $t-1$ colours are both necessary and sufficient to achieve a colouring with clustering $O_H(1)$  \cite{DBLP:conf/focs/DujmovicEMW23}.

Although the bound $O(n^{4/9})$ is less natural-looking than the $\Omega(n^{1/3})$ lower bound for planar graphs and the previous $O(n^{1/2})$ upper bound, there is reason to believe that the true bound for $H$-minor-free is somewhat mysterious: \citet{LINIAL_MATOUŠEK_SHEFFET_TARDOS_2008} showed that there exist $n$-vertex graphs with no $K_{6}$ minor such that every $3$-colouring has a cluster of size $\Omega(n^{4/10})$.  

\subsection{Proof Overview}

To prove \cref{genus_main_theorem} and \cref{minor_free_main_theorem}, we use some of the techniques of \citet{dujmović20253colouringplanargraphs} but often in different orders and interleaved in different ways.  A standard trick used for dealing with a bounded-genus graph $G$ is to find a \defin{planarizer} $P\subseteq V(G)$ of size $O(\sqrt{n})$ such that $G-P$ is planar.  Then one applies a result for planar graphs on $G-P$ and hopes that the additional $O(\sqrt{n})$ vertices in $P$ can be dealt with efficiently.  This fails when trying to establish $o(\sqrt{n})$ bounds for clustered colouring because the planarizer $P$ has very little structure and may have size $\Omega(\sqrt{n})$.

A $q$-separator in a graph $G$ is a subset $S$ of vertices of $G$ such that each component of $G-S$ has at most $q$ vertices.  A consequence of the Planar Separator Theorem is that every $n$-vertex planar graph has a $q$-separator of size $O(n/\sqrt{q})$.
One of the crucial observations in \cite{dujmović20253colouringplanargraphs} is that inclusion-minimal $q$-separators in planar graphs induce graphs of smaller treewidth than what is guaranteed by the Planar Separator Theorem.  In particular, they show that, for an inclusion-minimal $q$-separator $S$, the treewidth of $G[S]$ is $O(\sqrt{n/q})$, independent of the size of $S$. In contrast, the Planar Separator Theorem only shows that the treewidth of $G[S]$ is $O(\sqrt{|S|})\subseteq O(\sqrt{n}/q^{1/4})$.  We show that this result also holds for graphs of bounded genus using a proof that does not making use of planarizers.  

This low-treewidth separator result then implies that a bounded-genus graph $G$ has an $O(n^{4/9})$-separator $S$ of size $O(n^{7/9})$ such that $\tw(G[S])=O(n^{5/18})$. Each component of $G-S$ has size at most $O(n^{4/9})$ so we colour each component of $G-S$ red.  This leaves the colours green and blue for the graph $G[S]$.  The graph $G[S]$ is still a bounded genus graph, but now has a planarizer $P$ of size $O(\tw(G[S]))\subseteq O(n^{5/18})$.  The graph $G[S]-P$ is a planar graph, which allows us to use (generalizations of) the argument in \cite{dujmović20253colouringplanargraphs} to colour $G[S]-P$ blue and green so that each green component has size at most $O(n^{4/9})$ and the set $B$ of blue vertices has size at most $O(n^{4/9})$. We can then finish by colouring the vertices in $P$ blue.  By construction, each red component (of $G-S$) has size $O(n^{4/9})$, each green component (of $G[S]-(P\cup B))$ has size $O(n^{4/9})$ and the number of blue vertices is only $|P\cup B|\in O(n^{4/9})$.  

Note that the size $O(n^{5/18})$ of the planarizer $P$ in the preceding argument was smaller than necessary.  The argument continues to work even with $|P|\in O(n^{4/9})$.

To establish \cref{minor_free_main_theorem}, we make use of the \emph{Graph Minor Structure Theorem} \cite{ROBERTSON200343}, which decomposes any $H$-minor-free graph into ``torsos'' that are ``near-embedded'' and such that any two torsos share only a small number of vertices that form a clique in each torso.  Very roughly, near-embedded means that most of the complexity of a torso (in particular its treewidth) comes from a subgraph of bounded genus.  With some care, this allows us to use most of the machinery that we have developed for bounded-genus graphs. We partition the decomposition into parts where each part contains at least $n^{4/9}$ vertices of $G$.  Mostly, we can apply the preceding argument individually to each part focusing on one bounded-genus graph in one torso of the part.  When we do this, each red component is of size $O(n^{4/9})$ and is restricted to one part. Each green component is of size $O(n^{4/9})$, but may contain vertices from different parts.  The total size of all blue sets $B$ that we obtain from all the parts is $O(n^{4/9})$.  However, an issue arises from the planarizers that we get from different parts. There may be $\Omega(n^{5/9})$ parts, which means we cannot afford to include a constant number of vertices in the planarizer for each part.  This requires us to set a threshold for the sizes of parts and planarize only the sufficiently large parts.  Somewhat miraculously, we are able to choose a threshold that allows us to limit the total size of the planarizers to $O(n^{4/9})$.

The remainder of this paper is organized as follows:  \cref{sec::background} reviews some background material.  \Cref{sec::proof_thm1} proves \cref{genus_main_theorem}. \Cref{sec::h_minor_free} reviews the Graph Minor Structure Theorem before proving \cref{minor_free_main_theorem}.

\section{Background}\label{sec::background}

In this paper every graph $G$ that we consider is undirected and simple, with vertex set \defin{$V(G)$} and edge set \defin{$E(G)$}.  For a graph $G$ and a set $S$, \defin{$G[S]$} is the subgraph of $G$ induced by the vertices in $S\cap V(G)$. That is, $\mathdefin{V(G[S])}:=S\cap V(G)$ and $\mathdefin{E(G[S])}:=\{vw\in E(G):v,w\in S\}$. For a rooted tree $T$ and a node $x\in V(T)$, $T_x$ denotes the subtree of $T$ that includes $x$ and all descendants of $x$.

\subsection{\boldmath Treewidth, Grids, and \texorpdfstring{$q$-Separators}
{q-Separators}}\label{treewidth_grid_q-separation}
A \defin{tree decomposition} of a graph $G$ is a collection $\mathcal{T}:=(B_x:x\in V(T))$ of vertex subsets of $G$, called \defin{bags}, that is indexed by the vertices of a tree $T$ such that:
\begin{enumerate}[nosep,nolistsep,label=(\roman*)]
	\item\label{covers_edges} for each $vw\in E(G)$, there exists $x\in V(T)$ such that $\{v,w\}\subseteq B_x$; and
	\item \label{connectivity} for each vertex $v$ of $G$, $T[\{x\in V(T): v\in B_x\}]$ is a non-empty (connected) subtree of $T$.
\end{enumerate}
A \defin{path-decomposition} of a graph $G$ is a tree-decomposition $\calP := (B_x : x\in V(P))$ where the underlying tree $P$ is a path.
The \defin{width} of a tree decomposition $\mathcal{T}:=(B_x:x\in V(T))$ is $\mathdefin{\width(\mathcal{T})}:=\max\{|B_x|:x\in V(T)\}-1$. The \defin{treewidth} of a graph $G$ is the minimum width of a tree-decomposition of $G$, taken over all tree-decompositions of $G$.

Let $G$ be a graph and $\mathcal{T}:=(B_x:x\in V(T))$ a tree-decomposition of $G$.  For each edge $xy$ of $T$, $B_x\cap B_y$ is an \defin{adhesion} of $\mathcal{T}$ incident to $x$ and $y$.  When $T$ is rooted and $x$ is the parent of $y$ in $T$, then $\mathdefin{\partial_\calT(y)}:=B_x\cap B_y$ is the \defin{parent adhesion} of $\mathcal{T}$ at $y$. If $x$ is the root of $T$ then $\mathdefin{\partial_\calT(x)}:=\emptyset$.  For each node $x$ of $T$, the \defin{torso} \defin{$\torso{G}{B_x}$} of $\mathcal{T}$ at $x$ is the graph obtained from $G[B_x]$ by adding every edge $vw$ such that $\{v,w\}\subseteq B_x\cap B_y$ for some $y\in N_T(x)$.  Thus, each adhesion of $\mathcal{T}$ at $x$ is a clique in $\torso{G}{B_x}$.  We make use of the  following lemma, which follows easily from the definition of treewidth and Helly's Theorem on trees:

\begin{lemma}\label{lem:treewidth_bound_tree_decomp}
	Let $\mathcal{T}$ be a tree-decomposition of a graph $G$.  If each torso of $\mathcal{T}$ has treewidth at most $t$, then $G$ has treewidth at most $t$.
\end{lemma}

For $q\ge 1$, a set $S$ of vertices in a graph $G$ is a \defin{$q$-separator} if every component of $G-S$ has at most $q$ vertices.  Variations of the following lemma are well-known (see, for example \cite{robertson1984graphminorsII,pstSublinearSep}). We include a proof here for the sake of completeness.

\begin{lemma}\label{lem:tw_separator_bound}
	For every integer $n\ge 1$, and every integer $q\ge 1$, every $n$-vertex graph $G$ with treewidth at most $t$ has a $q$-separator of size at most $(t+1)(\lceil n/q\rceil-1)$. 
\end{lemma}

\begin{proof}
	The proof is by induction on $n$.  If $n \le q$, then $S=\emptyset$ is a $q$-separator of $G$ with $|S|=0\le (t+1)(\lceil n/q\rceil-1)$.  We now assume that $n>q$.

	Let $\calT:=(B_x:x\in V(T))$ be a tree-decomposition of $G$ of width $t$, and root $T$ arbitrarily.
	Say that a node $x\in V(T)$ is \defin{$q$-large} if $|\bigcup_{y\in V(T_x)}B_y|> q$. Since $n>q$, the root of $T$ is $q$-large. Let $x$ be a $q$-large node of $T$ such that no strict descendant of $x$ is $q$-large. Let $G_x:=G[\bigcup_{z\in V(T_x)}B_z]$ and let $G^-:=G-V(G_x)$. If $V(G^-) = \emptyset$, then $B_x$ is a $q$-separator of $G$ of size $|B_x| \le t+1 \le (t+1)(\lceil n/q \rceil - 1)$. Otherwise, by the inductive hypothesis, $G^-$ has a $q$-separator $S^-$ of size at most $(t+1)(\lceil|V(G^-)|/q\rceil-1)\le (t+1)(\lceil|V(G)|/q\rceil-2)$. Then $S:=S^-\cup B_x$ is a $q$-separator of $G$ of size $|S^-|+|B_x|\le (t+1)(\lceil|V(G)|/q\rceil-1)$.
\end{proof}

\cref{lem:tw_separator_bound} is useful for extracting $q$-separators from tree-decompositions.  However, in many cases, the treewidth of a graph is a function of its number of vertices and this can be taken advantage of.  A \defin{graph class} is a set of graphs.  A graph class $\calG$ is \defin{hereditary} if it closed under taking induced subgraphs; that is for every $G\in\calG$ and $S\subseteq V(G)$, $G[S]\in\calG$.  A \emph{class} of graphs $\mathcal{G}$ has \defin{treewidth} $f(n)$ if, for every $n\in\N$ and every $n$-vertex graph $G$ in $\mathcal{G}$, $\tw(G)\le f(n)$. We make use of the following result.\footnote{\cref{q_separator_hereditary} is a special case of a more general result that considers hereditary graph classes in which every $n$-vertex graph has treewidth at most $c n^{\delta}$, for some $0<\delta<1$.}

\begin{lemma}[\citet{pstSublinearSep,harmoniousColorings}]\label{q_separator_hereditary}
	For every $a>0$ there exists $c>0$ such that if a hereditary graph class $\mathcal{G}$ has treewidth $a\sqrt{n}$ then, for every $n\in\N$ and every $q\ge 1$, every $n$-vertex graph in $\mathcal{G}$ has a $q$-separator of size at most $cn/\sqrt{q}$.
\end{lemma}

For instance, the class of planar graphs has treewidth $O(\sqrt{n})$, and thus every $n$-vertex planar graph has a $q$-separator of size $O(n/\sqrt{q})$.

\subsection{Models and Minors}

A \defin{model} $\mathcal{M}:=\{G_x:x\in V(H)\}$ of a graph $H$ in a graph $G$ is a set of pairwise vertex-disjoint connected subgraphs of $G$, indexed by the vertices of $H$, and such that $G$ contains at least one edge $vw$ with $v\in V(G_x)$ and $w\in V(G_y)$ for each edge $xy$ of $H$. A graph $H$ is a \defin{minor} of a graph $G$ if and only if $G$ has a model of $H$. The subgraphs of $G$ in a model $\mathcal{M}$ are called the \defin{branch sets} of $\mathcal{M}$. Treewidth is a \defin{minor-monotone} graph parameter, meaning that $\tw(H)\le \tw(G)$ if $H$ is a minor of $G$.  A graph $G$ is \defin{$H$-minor-free} if and only if $H$ is not a minor of $G$.

For every integer $k\ge 1$, the \defin{$k\times k$ grid}, denoted \defin{$\grid_k$}, is the graph on the vertex set $\{1,\ldots,k\}^2$ such that $(i, j)$ and $(i', j')$ are adjacent if and only if $|i - i'| + |j-j'| = 1$. For a graph $G$, $\mathdefin{\gm(G)}$ is the maximum integer $k$ such that $\grid_k$ is a minor of $G$.  There is a close relationship between $\gm(G)$ and $\tw(G)$.  For every integer $k\ge 1$, $\tw(\grid_k) = k$.  Since treewidth is minor-monotone, this implies $\tw(G)\ge \gm(G)$.  Lower bounding $\gm(G)$ by some function of $\tw(G)$ is an area of active research, that is still not fully resolved for general graphs.  However, for $H$-minor-free graphs, it is known that $\gm(G)$ is lower-bounded by a linear function of $\tw(G)$.

\subsection{Graphs on Surfaces}\label{graph_surfaces}

A \defin{surface} $\Sigma$ is a $2$-manifold without boundaries. More precisely, it is a compact connected Hausdorff topological space such that every point in $\Sigma$ is locally homeomorphic to the plane. Every surface $\Sigma$ is homeomorphic to a surface obtained from the sphere $\mathbb{S}^2$ by adding handles and cross-caps. The \defin{Euler-genus} \defin{$\eps(\Sigma)$} of a surface $\Sigma$ obtained by adding $h$ handles and $k$ cross-caps to $\mathbb{S}^2$ is  $2h+k$. An \defin{open arc} in $\Sigma$ is a subset of $\Sigma$ homeomorphic to the open interval $(0, 1)$. A \defin{loop} in $\Sigma$ is a subset of $\Sigma$ homeomorphic to a circle. We let $\mathdefin{\calA_{\Sigma}}$ be the set of all open arcs in $\Sigma$. The closure of a subset $X \subseteq \Sigma$ is denoted by $\mathdefin{\overline{X}}$, and the boundary $\mathdefin{\partial X}$ is equal to the intersection of $\overline{X}$ and $\overline{\Sigma - X}$. For an open arc $A\in\mathcal{A}_\Sigma$, $\partial(A)$ contains exactly two points of $\Sigma$ that are called the \defin{endpoints} of $A$.

An \defin{embedding} of a graph $G$ on a surface $\Sigma$ is a function from $\sigma:V(G)\cup E(G)\to\Sigma \cup \calA_{\Sigma}$ such that
\begin{enumerate}[nosep,nolistsep]
	\item $\sigma(v)\in\Sigma$ is a point in $\Sigma$, for each $v\in V(G)$;
	\item $\sigma(v)\neq\sigma(w)$ for all distinct $v,w\in V(G)$;
	\item $\sigma(vw)\in\calA_\Sigma$ is an arc in $\Sigma$ with endpoints $\sigma(v)$ and $\sigma(w)$, for each $vw\in E(G)$;
	\item $\sigma(vw)\cap\sigma(V(G)) = \{\sigma(v),\sigma(w)\}$ for each $vw\in E(G)$; and
	\item $\sigma(e_1)\cap\sigma(e_2)=\emptyset$ for all distinct $e_1,e_2\in E(G)$.
\end{enumerate}
The \defin{faces} of an embedding $\sigma$ are the connected components of $\Sigma\setminus \sigma(G)$. 
We say an embedding $\sigma$ is a \defin{2-cell} embedding if every face of $\sigma$ is homeomorphic to an open disc in $\R^2$. 
The \defin{Euler-genus of a graph} $G$, denoted by $\mathdefin{\eps(G)}$, is the minimum $\eps(\Sigma)$ over all the surfaces $\Sigma$ such that $G$ has a $2$-cell embedding.\footnote{The restriction to $2$-cell embeddings is for convenience only.  If a graph $G$ has an embedding on a surface of Euler genus $g$ then it has a $2$-cell embedding on a surface of Euler genus at most $g$.}  Euler genus is also a minor-monotone property: If $H$ is a minor of $G$ then $\eps(H)\le\eps(G)$. A \defin{planar graph} is a graph of Euler-genus $0$.

An \defin{embedded graph} is a graph $G$ equipped with an embedding $\sigma$ of $G$ in a surface $\Sigma$.  When there is no danger of ambiguity, we shall not distinguish between an embedded graph $G$ and its embedding. In other words, each vertex $v$ of $G$ is a point $\sigma(v)$ in $\Sigma$, and each edge is an arc $\sigma(e)$ in $\Sigma$.  The set of faces of an embedded graph $G$ with embedding $\sigma$, denoted \defin{$F(G)$}, is the set of the faces of $\sigma$.  A \defin{plane graph} is a graph embedded on a surface of Euler genus $0$.

\subsection{Face-Weighted Embedded Graphs}

Let $G$ be an embedded graph.  For a vertex $v$ of $G$, we define $\mathdefin{F(G, v)}$ to be the set of faces of $G$ incident to $v$. A \defin{face-weighting} of an embedded graph $G$ is a function $w: A \to \R^+$ where $A \subseteq F(G)$. For any $B \subseteq A$, $w(B)$ denotes $\sum_{f \in B} w(f)$.  As a first step in their proof, \citet{dujmović20253colouringplanargraphs} show that face-weightings can be used to upper bound the treewidth of plane graphs.  In this lemma, the set $Y$ represents a set of faces whose weight is undefined and the lemma requires that weights of the faces incident to a vertex $v$ sum to at least $q$, unless $v$ is incident to a face in $Y$.

\begin{lemma}[\citet{dujmović20253colouringplanargraphs}]\label{planar_tw_by_face_weights}
	Let $N\ge q\ge 1$ be real numbers, let $G$ be a plane graph, let $Y\subseteq F(G)$, and let $w: F(G)\setminus Y\to \R^+$ be a face-weighting such that $w(F(G)\setminus Y) \leq N$ and $F(G, v) \cap Y \neq \emptyset$ or $w(F(G, v) \setminus Y) \geq q$ for each $v \in V(G)$.  Then
	$$\tw(G) \leq 12 \sqrt{|Y| + \frac{N}{q}} + 7 \enspace . $$
\end{lemma}

We need a generalization of \cref{planar_tw_by_face_weights} for graphs embedded on surfaces of Euler genus $g>0$.  The proof of \cref{planar_tw_by_face_weights} uses the fact that, for any plane graph $G$, $\gm(G)\in \Omega(\tw(G))$.  We need a generalization of this result to graphs embedded on surfaces of Euler genus $g>0$.  
We make use of the following result, which provides precise bounds:\footnote{This appears as \cite[Theorem 4.12]{DBLP:journals/jacm/DemaineFHT05}, which is phrased in terms of branch-width $\bw$ instead of the treewidth. However, for a graph $G$, $\bw(G) - 1\leq \tw(G) \leq \lfloor\frac{3}{2}\bw(G)\rfloor - 1.$}

\begin{theorem}[\citet{DBLP:journals/jacm/DemaineFHT05}]\label{excluded_grid_genus}
	If $G$ is a graph of Euler genus $g$ with treewidth greater than $6(g+1)r - 1$, then $G$ has $\grid_r$ as a minor.
\end{theorem}

Note that, since $\grid_r$ has $r^2$ vertices, any graph that has a $\grid_r$-minor has at least $r^2$ vertices.  That is, $\gm(G)\le \sqrt{n}$ for any $n$-vertex graph $G$. The contrapositive of \cref{excluded_grid_genus} states that $\tw(G)\le 6(g+1)\gm(G)$ for Every genus-$g$ graph $G$.  Combining these two inequality implies that $\tw(G)\le 6(g+1)\gm(G)\le 6(g+1)\sqrt{n}$ for every $n$-vertex graph of genus at most $g$.  Thus, for every $g\ge 1$, every $n$-vertex genus-$g$ graph has $\tw(G)\in O_g(\sqrt{n})$.  

\Cref{excluded_grid_genus} is still not enough.  In a plane embedding of $\grid_r$, for $r\ge 3$, each of the $4$-cycles of $\grid_r$ bounds a distinct face of $\grid_r$. This is not necessarily true for graphs embedded on surfaces of Euler-genus $g\ge 1$.  The following result shows that it is true ``often enough'' for our purposes:

\begin{lemma}[\citet{GEELEN2004785}]\label{galeen_and_elt}
	Let $r\ge 3$, let $\Sigma$ be a surface of Euler-genus $g$, and let $G=\grid_r$ be embedded in $\Sigma$. Then the  number of non-contractible cycles of length four in $G$ is at most $9g$.
\end{lemma}

We can now prove the generalization of \cref{planar_tw_by_face_weights} that we need:
\begin{lemma}\label{surface_tw_by_face_weights}
	Let $G$ be a graph embedded on a surface $\Sigma$ of Euler genus $g$, let $Y\subseteq F(G)$ be a set of faces in $G$, and let $w:F(G)\setminus Y\to \mathbb{R}^+$ be a face-weighting of $G$ with $w(F(G)\setminus Y)\le N$ and such that, for each $v\in V(G)$, either $F(G,v)\cap Y \neq \emptyset$ or $\sum_{f\in F(G,v)} w(f)>q$. Then $\tw(G)\in O_g(\sqrt{N/q+|Y|})$.
\end{lemma}

\begin{proof}
	If $g=0$, then the lemma is already implied by \cref{planar_tw_by_face_weights}, so we may assume that $g\ge 1$.  In particular, this implies that any contractible loop in $\Sigma$ is the boundary of exactly one disc in $\Sigma$. We may assume that $\gm(G)\ge 3$ since, otherwise $\tw(G)\le 18g+17\in O_g(\sqrt{N/q+|Y|})$ since $N/g\ge 1$.  
    
    Let $r:=\gm(G)$ and let $\calM:=\{G_x:x\in V(\grid_r)\}$.  Without loss of generality, we may assume that each branch set of $\mathcal{M}$ is a tree, since each branch set $G_x$ can be replaced with a spanning tree of $G_x$.  For each edge $xy$ of $\grid_r$, choose an edge $G_{vw}$ of $G$ with $v\in V(G_x)$ and $w\in V(G_y)$. Let $M:=\bigcup_{x\in V(\grid_r)}G_x \cup \bigcup_{e\in E(\grid_r)}G_e$.	Then each cycle $x_0,\ldots,x_c$ in $\grid_r$ corresponds to a cycle $G_{C}:=P_0,G_{x_0x_1},P_1,G_{x_1x_2},\ldots,P_c,G_{x_cx_0}$ in $M$, where $P_i$ is the unique path in $G_{x_i}$ with endpoints $V(G_{x_{i-1}x_i})\cap V(G_{x_i})$ and $V(G_{x_ix_{i+1}})\cap V(G_{x_i})$ (with subscripts taken modulo $c+1$). 

	Say that a $4$-cycle $C$ in $\grid_r$ is \defin{genus-bad} if $G_C$ is non-contractible and is \defin{genus-good} otherwise.  By \cref{galeen_and_elt}, at most $9g$ $4$-cycles in $\grid_r$ are genus-bad.  Note that, for any two distinct genus-good $4$-cycles $C_1$ and $C_2$ in $\grid_r$, the open discs in $\Sigma$ with boundary $G_{C_1}$ and $G_{C_2}$ are disjoint. Let $Z$ be a set of $\lfloor (r-1)/2\rfloor^2$ degree-$4$ vertices of $\grid_r$ that are not incident to any common faces of $\grid_r$. Say that a vertex $x \in Z$ is \emph{genus-bad} if any of the four $4$-cycles incident to $x$ in $\grid_r$ are genus-bad, and $x$ is \defin{genus-good} otherwise. Since there are at most $9g$ genus-bad $4$-cycles in $\grid_r$ and the vertices in $Z$ do not share any incident $4$-cycles, at most $9g$ vertices in $Z$ are genus-bad.

	For each genus-good node $x$ in $Z$, select an arbitrary vertex $v_x$ in $G_x$. Observe that, for each face $f$ of $G$ incident to $v_x$, $\grid_r$ has a (genus-good) $4$-cycle $C$ incident to $x$ such that the disc in $\Sigma$ whose boundary is $G_C$ contains $f$.  Say that a genus-good vertex $x \in Z$ is \defin{$Y$-bad} if $F(G,v_x)$ contains a face in $Y$, and say that $x$ is \defin{$Y$-good} otherwise. Because the sets $F(G,v_x)$ are pairwise disjoint for all genus-good vertices in $Z$, at most $|Y|$ genus-good vertices in $Z$ are $Y$-bad. Let $Z'$ be the set of vertices in $Z$ that are genus-good and $Y$-good. Then $|Z'| \ge \lfloor (r-1)/2\rfloor^2 - 9g - |Y|$. 

	For each $x\in Z'$, $F(G,v_x)\cap Y = \emptyset$, so $\sum_{f\in F(G,v_x)} w(f) > q$ for each $x\in Z'$. By the choice of $Z'$, $F(\grid_r,x)$ and $F(\grid_r,y)$ are disjoint for each distinct $x,y\in Z'$. Therefore, $F(G,v_x)$ and $F(G,v_y)$ are disjoint for distinct $x,y\in Z'$. Thus,
	\[
		N
		 \ge \sum_{x\in Z'} \sum_{f\in F(G,v_x)} w(f)
		 > |Z'|q
		 \ge (\lfloor (r-1)/2\rfloor^2-9g-|Y|)q \enspace .
	\]
	Rewriting this inequality gives $r< 2\sqrt{N/q+9g+|Y|}$. By \cref{excluded_grid_genus}, $\tw(G)\le 6(g+1)r-1\le 6(g+1)\left( 2\sqrt{N/q+9g+|Y|}\right)-1$, so $\tw(G)\le O(g\sqrt{N/q+g+|Y|})$.
\end{proof}

\subsection{Planarizers}

Let $\Sigma$ be a surface.
A loop $C$ is \defin{contractible} in $\Sigma$ if it is null homotopic in $\Sigma$; that is, we can continuously deform $C$ in $\Sigma$ to a point. Otherwise, we say $C$ is \defin{non-contractible}. For a 2-cell embedded graph $G$ on $\Sigma$, a \defin{noose} is a loop that intersects the embedding only at the vertices of $G$. The \defin{length} of a noose is the number of vertices it intersects. 

\begin{theorem}[\citet{DBLP:journals/jacm/DemaineFHT05}]\label{cor::shortest_noncontractible_noose}
	Let $r$ be a positive integer. If a graph $G$ is $2$-cell embedded on a surface $\Sigma$ of Euler genus greater than $0$ in such a way that every non-contractible noose has length at least $4r$, then $\grid_r$ is a minor of $G$.
\end{theorem}

\defin{Cutting} a surface $\Sigma$ along a loop $C$ is equivalent to the operation $\Sigma \setminus C$, that is, removing the points of $C$ from $\Sigma$. If the loop $C$ is contractible, the cutting operation produces an open disk $D$, and a surface $\Sigma'$, each with boundary equal to $\partial D$. Otherwise, $C$ is non-contractible, and cutting along $C$ produces one or two surfaces each of genus smaller than that of $\Sigma$. Precisely, if $C$ is one-sided, we obtain one surface, otherwise it is two-sided and we obtain at most two surfaces. \defin{Slicing} a surface $\Sigma$ along a loop $C$ is equivalent to cutting along $C$, then capping each hole by a closed disk; by capping we mean identifying the boundary of the hole to the boundary of a disk. When slicing along a non-contractible loop $C$ in a surface $\Sigma$, there are three possibilities:
\begin{itemize}[nosep,nolistsep]
	\item If $C$ is one-sided then we obtain one surface $\Sigma_1$ with $\eps(\Sigma_1)=\eps(\Sigma)-1$.
	\item If $C$ is two-sided and non-surface-separating then we obtain one surface $\Sigma_1$ with $\eps(\Sigma_1)=\eps(\Sigma)-2$ such that $\Sigma\setminus \Sigma_1$ consists of two discs.
	\item If $C$ is two-sided and surface-separating, then we obtain two surfaces $\Sigma_1$ and $\Sigma_2$ such that $\eps(\Sigma_1)+\eps(\Sigma_2)=\eps(\Sigma)$ and  $\Sigma_i\setminus \Sigma$ is a disc, for each $i\in\{1,2\}$.
\end{itemize}

A \defin{planarizer} $P$ of a graph $G$ is a set of vertices of $G$ such that $G\setminus P$ is a planar graph.
In the next lemma, we show that a graph $G$ that has an embedding $\sigma$ on a surface of genus $g$ has a planarizer whose size is $O(g\cdot\tw(G))$.  The proof proceeds by repeated applications of \cref{cor::shortest_noncontractible_noose}.  Although there are a number of similar proofs in the literature, we provide a detailed proof here that also shows the existence of a plane embedding of $G-P$ that has a small number of faces that are not faces of $G$.  This is important in subsequent applications of this lemma because these faces will become the set $Y$ in applications of \cref{surface_tw_by_face_weights}.

\begin{lemma}\label{planarizer_tw}
	Let $G$ be a graph and let $\sigma$ be an embedding of $G$ on a surface $\Sigma$ of Euler genus $g\ge 1$.  Then there exists a planarizer $P$ of $G$ of size at most $4(2g-1)(\tw(G)+1)$ and an embedding $\sigma'$ of $G-P$ onto a surface $\Sigma'$ of genus $0$ such that $\sigma'$ has at most $2g-1$ faces that are not faces of $\sigma$.
\end{lemma}

\begin{proof}
	We proceed by induction on $g$.  The two base cases $g=1$ and $g=2$ are handled below.

	If $\sigma$ is not a $2$-cell embedding, then some face of $\sigma$ is not homeomorphic to an open disk and thus contains a non-contractible loop in $\Sigma$. Such a loop is a non-contractible noose of length $0$. Otherwise, $\sigma$ is a $2$-cell embedding. In either case, let $\calN$ be the shortest non-contractible noose in $G$ with respect to $\sigma$ and let $P_0:=\calN\cap V(G)$. If $\sigma$ is a $2$-cell embedding, \cref{cor::shortest_noncontractible_noose} implies $|P_0|\le 4(\tw(G)+1)$; otherwise $|P_0|=0 \le 4(\tw(G)+1)$.  We distinguish several cases:
	\begin{itemize}
		\item If $\calN$ is one-sided, then slicing $\Sigma$ along $\calN$ produces one surface $\Sigma_1\supseteq\Sigma$ of genus $g-1$.  Then $\Sigma_1\setminus\Sigma$ is a disc $D_1$ bounded by $\calN$. Let $G_1:=G-P_0$.  The restriction $\sigma_1$ of $\sigma$ to $G_1$ is an embedding of $G_1$ on $\Sigma_1$ in which $F(G_1)\setminus F(G)=\{f_1\}$, where $f_1$ is the face of $G_1$ that contains $D_1$.  If $g=1$, then $P:=P_0$, $\sigma':=\sigma_1$, and $\Sigma':=\Sigma_1$ satisfies the requirements of the lemma.  Assume $g>1$.  
		
		Since $\tw(G_1)\le \tw(G)$, we can apply induction to $(G_1,\sigma_1,\Sigma_1)$ to obtain a planarizer $P_1$ of $G_1$ of size at most $4(2(g-1)-1)(\tw(G)+1)$ and an embedding $\sigma_1'$ of $G_1-P_1$ onto a surface $\Sigma_1'$ of genus $0$ such that $|F(G_1-P_1)\setminus F(G_1)|\le 2(g-1)-1$.
		Then $P:=P_0\cup P_1$ has size at most $|P_0|+|P_1|\le 4(2g-2)(\tw(G)+1)$, $\sigma':=\sigma_1'$ is an embedding of $G-P$ onto $\Sigma':=\Sigma_1'$ and $F(G-P)\subseteq (F(G_1-P_1)\setminus F(G_1))\cup\{f_1\}$, so $|F(G-P)\setminus F(G)|\le (2(g-1)-1)+1=2g-2$. 
			
		\item If $\calN$ is two-sided and non-surface separating, then slicing $\Sigma$ along $\calN$ produces one surface $\Sigma_1$ of genus $g-2$. Then $\Sigma_1\setminus\Sigma=D_1\cup D_2$ where $D_1$ and $D_2$ are disjoint discs. Let $G_1:=G-P_0$.  The restriction $\sigma_1$ of $\sigma$ to $G_1$ is an embedding of $G_1$ on $\Sigma_1$ in which $F(G_1)\setminus F(G)=\{f_1,f_2\}$, where $f_i$ is the face of $G_1$ that contains $D_i$, for each $i\in\{1,2\}$. 	If $g=2$, then $P:=P_0$, $\sigma':=\sigma_1$, and $\Sigma':=\Sigma_1$ satisfies the requirements of the lemma. Assume $g>2$.  
		
		Since $\tw(G_1)\le\tw(G)$, we can apply induction to $(G_1,\sigma_1,\Sigma_1)$ to obtain a planarizer $P_1$ of $G_1$ of size at most $4(2(g-2)-1)(\tw(G)+1)$ and an embedding $\sigma_1'$ of $G_1-P_1$ onto a surface $\Sigma_1'$ of genus $0$ such that $|F(G_1-P_1)\setminus F(G_1)|\le 2(g-2)-1$.  Then $P:=P_0\cup P_1$ has size at most $|P_0|+|P_1|\le 4(2g-4)(\tw(G)+1)$ and $\sigma':=\sigma_1'$ is an embedding of $G-P$ onto $\Sigma':=\Sigma_1'$ and $|F(G-P)\setminus F(G)|\le (2(g-2)-1)+2 = 2g-3$.

		\item If $\calN$ is two-sided and surface separating, then slicing $\Sigma$ along $\calN$ produces two surfaces $\Sigma_1$ and $\Sigma_2$ of genera $g_1$ and $g_2$ respectively, such that $g_1+g_2=g$.  Since $\calN$ is non-contractible in $\Sigma$, $g_i\ge 1$ for each $i\in\{1,2\}$.  For each $i\in\{1,2\}$, $\Sigma_i\setminus\Sigma$ is a disc $D_i$ bounded by $\calN$. For each $i\in\{1,2\}$, let $G_i$ be the subgraph of $G-P_0$ that $\sigma$ embeds onto $\Sigma_i\setminus D_i$.  Then, for each $i\in\{1,2\}$, $F(G_i)\setminus F(G)=\{f_i\}$, where $f_i$ is the face of $G_i$ that contains $D_i$.
		
		For each $i\in\{1,2\}$, applying induction on $(G_i,\sigma_i,\Sigma_i)$ yields $P_i\subseteq V(G_i)$, $\sigma_i'$, and $\Sigma_i'$ with $|P_i|\le 4(2g_i-1)(\tw(G)+1)$ and $|F(G_i-P_i)\setminus F(G_i)|\le 2g_i-1$. Let $P:=P_0\cup P_1\cup P_2$.  Then
		\begin{align*}
		  |P| & \le 4(1 + (2g_1-1) + (2g_2-1))(\tw(G)+1) \\
			    & = 4(2(g_1+g_2) -1)(\tw(G)+1) = 4(2g-1)(\tw(G)+1) \enspace .
		\end{align*}
		For each $i\in\{1,2\}$, the disc $D_i$ is contained in some face $f'_i \in F(G_i-P_i)$. We let $\Sigma'$ be the surface obtained by gluing $\Sigma_1'\setminus D_1$ and $\Sigma_2'\setminus D_2$ along their common boundary $\calN$. Then $\Sigma'$ is a connected surface of genus $0$. The embeddings $\sigma_1'$ and $\sigma_2'$ naturally define an embedding $\sigma'$ of $G-P$ onto $\Sigma'$. The gluing operation merges $f'_1$ and $f'_2$ into a single face $f'_{12}$ in $G-P$. Therefore, 
		\[ F(G-P) = (F(G_1-P_1) \setminus \{f'_1\}) \cup (F(G_2-P_2) \setminus \{f'_2\}) \cup \{f'_{12}\} \enspace . \]
		The faces of $G-P$ that are not in $F(G)$ consist of at most the new faces of $G_1-P_1$ and $G_2-P_2$ and $f'_{12}$ (if $f_1$ and $f_2$ are each present in $G_1-P_1$ and in $G_2-P_2$ respectively). Thus,
		\begin{align*}
			|F(G-P)\setminus F(G)|& \le |F(G_1-P_1)\setminus F(G_1)| + |F(G_2-P_2)\setminus F(G_2)| + 1 \\
			& \le (2g_1-1)+(2g_2-1)+1 = 2g-1 \enspace . \qedhere 
		\end{align*}
	\end{itemize}
\end{proof}

\section{Proof of \texorpdfstring{\cref{thm::main}}{thm::main}}\label{sec::proof_thm1}

We now have most of the tools needed to prove \cref{genus_main_theorem}.  We make use of the following lemma, which appears in \cite{dujmović20253colouringplanargraphs}, but is closely related to various proofs and generalizations of the Planar Separator Theorem (for example the proof of the existence of cycle separators of size $O(\sqrt{n})$ in triangulations by \citet{alon.seymour.ea:planar}):

\begin{lemma}[\citet{dujmović20253colouringplanargraphs}]\label{finding_noose}
	For any plane graph $G$ and face-weighting $w: F(G)\to \R_{\geq 0}$ with total weight $N := w(F(G)),$ there exists a noose $\mathcal{N}$ of $G$ of length at most $\max\{4\tw(G), 12\} + 4$ and for which $w(F(G) \cap F(C)) \leq \frac{2}{3}N$ for each component $C$ of $G -\mathcal{N}$.
\end{lemma}

We remark that the most obvious generalizations of \cref{finding_noose} to graphs embedded on surfaces of Euler genus $g>0$ can only guarantee that the separator is made up of $O(g)$ nooses. This becomes (at the very least) difficult to reason about when the lemma is applied iteratively to the components of $G-\calN$.  

The last tool we use is the following generalization of \cite[Lemma~10]{dujmović20253colouringplanargraphs}.  The only addition to this version of the lemma is that it supports a (small) set $Y$ of faces whose weights are undefined.  This generalization is needed when handling face-weighted higher-genus surface-embedded graphs that have been planarized using \cref{planarizer_tw}.

\begin{lemma}\label{lem::reduce_tw}
	There exists a constant $c$ such that the following is true: Let $N, q, t, y$ be real numbers such that $y\ge 0$,  $N \geq q \geq 1$, $t\ge \sqrt{1+\log_{3/2} N + y}$, let $G$ be a plane graph, let $Y\subseteq F(G)$ with $|Y|\le y$, and let $w: F(G)\setminus Y \to \R_{\geq 0}$ be a face-weighting of $G$ with $w(F(G)\setminus Y) \leq N$ and such that, for each $v\in V(G)$, either $F(G,v)\cap Y \neq \emptyset$ or $\sum_{f\in F(G,v)} w(f)>q$.	Then, there exists $S \subseteq V(G)$ such that $|S| \leq cN/(tq)$ such that $\tw(G - S) \le ct$.
\end{lemma}

\begin{proof}
	This proof closely follows that of \cite[Lemma~10]{dujmović20253colouringplanargraphs} with the additional modifications needed to handle the special set $Y$ of faces whose weights are undefined.  This is not a large modification, since the proof in \cite{dujmović20253colouringplanargraphs} already needs to deal with the fact that repeated applications of \cref{finding_noose} immediately lead to components that have faces of undefined weight. The reader can safely skip this proof in a first reading.
	
	For each $f\in Y$, define $w(f):=0$.
	For each connected subgraph $C$ of $G$, define $N_C := w(F(C) \cap F(G))$, define $X_C:=F(C)\setminus F(G)$, and let $X^+_C:=X_C\cup (Y\cap F(C))$.  Initialize $S:= \emptyset$, and let $\calL$ be a queue containing the connected components of $G$. While $\calL$ is not empty, dequeue the next element $C$ from $\calL$, and do one of the following.
	\begin{enumerate}[nosep,nolistsep]
		\item If $N_C \le t^2q$ then do nothing.
		\item Otherwise, apply \cref{finding_noose} to find a noose $\calN_C$ for $C$ of length at most $\max\{4\tw(C),12\}+4$  such that each component $C'$ of $C-\calN_C$ has $N_{C'}\le \tfrac{2}{3}N_C$. Define $S_C:=\calN_C\cap V(C)$ and enqueue each component $C'$ of $C-S_C$ into $\calL$. 
	\end{enumerate}
	Let $\calC$ be the set of components $C$ that are dequeued from $\calL$ and for which $N_C > t^2q$.
	This process defines a rooted forest $F$ whose roots are the components of $G$, whose non-leaf nodes are the elements in $\calC$, and in which  $C'$ is a child of $C$ if and only if $C'$ is a component of $C-S_C$.  By construction, if a node $C$ has depth $d$ in $F$, then $N_C\le (\tfrac{2}{3})^d N$.  It follows that no node $C\in\calC$ has depth greater than $\log_{3/2} (N/(t^2q))$ in $F$ and that no node of $F$ has depth greater than $1+\log_{3/2} (N/(t^2q))$.  Furthermore, if $C'$ is a child of $C$ in $F$, then $X^+_{C'}\setminus X^+_C$ contains at most one face, namely the face of $C'$ that contains $\calN_C$.  It follows that, if $C$ has depth $d$ in $F$, then $|X^+_C|\le d+|Y|\le 1+\log_{3/2}N +y$.
	
	Let $S:=\bigcup_{C\in\cal C} S_C$.  Then each component $C$ of $G-S$ is a leaf in $F$ with $N_C \le t^2q$ so, by \cref{planar_tw_by_face_weights},
	\[
	  \tw(C) 
		  \le 12\sqrt{N_C/q + |X^+_C|}+7 
		\le 12\sqrt{t^2 + 1+\log_{3/2} N + |Y|}+7 
		  \le ct 
	\]
	for a sufficiently large constant $c$, since $t\ge \sqrt{1+\log_{3/2} N+y}$.  Thus, $\tw(G-S)\le ct$ since every component $C$ of $G-S$ has treewidth at most $ct$.

	All that remains is to bound the size of $S$. For each $i\in\N$, define the contour 
	\[
		\Delta_i:=\{C\in\calC: (\tfrac{3}{2})^i t^2q < N_C \le (\tfrac{3}{2})^{i+1} t^2q\}.
	\]
	If $C_1$ and $C_2$ are distinct elements in $\Delta_i$ then $C_2$ is not a descendant of $C_1$ in $F$ because every descendant $C'$ of $C_1$ has
	$N_{C'} \le (\tfrac{2}{3})N_{C_1} \le (\tfrac{2}{3})(\tfrac{3}{2})^{i+1} t^2q = (\tfrac{3}{2})^i t^2q < N_{C_2}$.
	In particular, this means that $F(C_1)\cap F(G)$ and $F(C_2)\cap F(G)$ are disjoint.  Therefore
	\[  |\Delta_i|(\tfrac{3}{2})^it^2q < \sum_{C\in\Delta_i} N_C \le N \quad \text{so}\quad 
	|\Delta_i| < (\tfrac{2}{3})^i(N/t^2q)
	\]
	for each $i\in\N$.

	For each $i\in\N$ and each $C\in\Delta_i$, by \cref{planar_tw_by_face_weights},
	\begin{align*}
		|S_C| &\le \max\{4\tw(C),12\}+4 \\ 
		  &\le \max\{4(12\sqrt{N_C/q + |X^+_C|}+7),12\}+4 \\
		  &\le \max\{4(12\sqrt{(\tfrac{3}{2})^{i+1}t^2+\log_{3/2} N + y}+7),12\}+4 \\
		  &\le (\tfrac{3}{2})^{(i+1)/2} ct 
	\end{align*}
	for a sufficiently large constant $c$.
	Therefore,
	\[
	  \sum_{C\in\Delta_i}|S_C|
		\le |\Delta_i| \cdot (\tfrac{3}{2})^{(i+1)/2}ct 
		\le (\tfrac{2}{3})^i(N/t^2q)\cdot (\tfrac{3}{2})^{(i+1)/2}ct 
		= (\tfrac{2}{3})^{(i-1)/2}\cdot cN/tq
	\]
  and
	\[
	  |S| = \sum_{i\ge 0}\sum_{C\in\Delta_i}|S_C| 
		\le \sum_{i\ge 0}(\tfrac{2}{3})^{(i-1)/2}\cdot cN/tq 
		\le c'N/tq
	\]
  for a sufficiently large constant $c'$.
\end{proof}

With these tools in place, the proof of \cref{genus_main_theorem} now happens quickly.

\begin{lemma}\label{reduce_treewidth_genus}
	Let $g\in\N$, let $N$ and $q$ be real numbers such that $N/q^2\ge \sqrt{1+\log_{3/2} N+2g}$, let $G$ be a graph embedded on a surface of Euler genus $g$, let $w:F(G)\to \mathbb{R}^+$ be a face-weighting of $G$  with $w(F(G)) \leq N$ and such that $\sum_{f\in F(G,v)} w(f)>q$ for each $v\in V(G)$. Then there exists $S\subseteq V(G)$ of size $O_g(q+\sqrt{N/q})$ such that $\tw(G-S)\in O_g(N/q^2)$.
\end{lemma}

\begin{proof}
	By applying \cref{surface_tw_by_face_weights} with $Y=\emptyset$ and the face weighting function $w$, we see that $\tw(G)\in O_g(\sqrt{N/q})$.  By \cref{planarizer_tw}, there exists a planarizer $P$ of $G$ of size $O_g(\tw(G))=O_g(\sqrt{N/q})$ and an embedding $\sigma'$ of $G-P$ on a surface of Euler genus $0$ such that $|F(G-P)\setminus F(G)|\le 2g-1$.  Let $Y:=F(G-P)\setminus F(G)$. Let $t=N/q^2$.  By \cref{lem::reduce_tw}, there exists $S'\subseteq V(G-P)$ of size $O(N/qt)=O(q)$ such that $\tw(G-(P\cup S'))\le t=N/q^2$.  Let $S:=P\cup S'$. Then $|S|\in O_g(q + \sqrt{N/q})$ and $\tw(G-S)\in O_g(N/q^2)$.
\end{proof}

\begin{lemma}\label{three_step_genus}
	Let $g\in\N$ and let $q,n\in\N$ such that $n/q^2\ge \sqrt{1+\log_{3/2} n + 2g}$, and let $G$ be an $n$-vertex graph embedded on a surface $\Sigma$ of Euler genus $g$.  Then there exists a $q$-separator $S$ of $G$ of size $O_g(n/\sqrt{q})$ and a $q$-separator $S'\subseteq S$ of $G[S]$ of size $O_g(q+\sqrt{n/q})$ such that $\tw(G[S]-S')\in O_g(n/q^2)$.
\end{lemma}

\begin{proof}
	By \cref{q_separator_hereditary} and (the paragraph following) \cref{excluded_grid_genus}, $G$ has a $q$-separator $S$ of size $O_g(n/\sqrt{q})$.  We may assume that $S$ is inclusion-minimal in the sense that the component of $G-(S\setminus\{v\})$ that contains $v$ contains more than $q$ vertices, for each $v\in S$. For each face $f\in F(G[S])$, let $w(f)$ be the number of vertices of $G-S$ embedded in the interior of $f$. Then, for each $v\in S$, the component of $G-(S\setminus\{v\})$ that contains $v$ is contained in $\{v\}\cup\bigcup_{f\in F(G[S],v)} f$.  Therefore, $1+\sum_{f\in F(G[S],v)}w(f)>q$, so $\sum_{f\in F(G[S],v)}w(f)\ge q$, for each $v\in S$. Let $N := w(F(G[S])) = |V(G-S)| \le n$.
	By \cref{reduce_treewidth_genus}, there exists $S'\subseteq S$ such that $|S'|\in O_g(q+\sqrt{n/q})$ and $\tw(G[S]-S')\in O_g(n/q^2)$.
\end{proof}

\begin{proof}[Proof of \cref{genus_main_theorem}]
	Fix an embedding of $G$ on a surface $\Sigma$ of Euler genus $g$.  Let $q:=n^{4/9}$, so that $t:=n/q^2=n^{1/9}$. If $n^{1/9}<\sqrt{1+\log_{3/2} n + 2g}$ then $n^{2/9}-\log_{3/2} n < 1+2g$ so $n\in O_g(1)$.  In this case we can colour all vertices of $G$ red and the resulting colouring has clustering $O_g(1)$. Now assume that $n^{1/9}\ge\sqrt{1+\log_{3/2} n + 2g}$.
    
    Apply \cref{three_step_genus} to obtain a $q$-separator $S$ of $G$ of size $O_g(n/\sqrt{q})=O_g(n^{7/9})$ and $S'\subseteq S$ of size $O_g(q+\sqrt{n/q})=O_g(n^{4/9})$ such that $\tw(G[S]-S')\in O_g(n/q^2)\le c(n^{1/9})$ for some $c$ that depends only on $g$. Apply \cref{lem:tw_separator_bound} to $G[S]-S'$ with $t=c n^{1/9}$ to obtain a $q$-separator $S''\subseteq S\setminus S'$ of $G[S]-S'$ with $|S''|\le(ct+1)|V(G[S]-S')|/q = O_g(n^{4/9})$.  Colour each component of $G-S$ red. Colour each component of $G[S]-(S'\cup S'')$ blue.  Colour the vertices of $S'\cup S''$ green.

    Each red component has size at most $n^{4/9}$ because $S$ is a $q$-separator of $G$.  Each green component has size at most $n^{4/9}$ because $S''$ is a $q$-separator of $G[S]-S'$, so $S''\cup S'$ is a $q$-separator of $G[S]$.  Each blue component has size $O(n^{4/9})$ because $|S'\cup S''|\in O(n^{4/9})$.
\end{proof}

\section{\boldmath \texorpdfstring{$H$}{H}-Minor-Free Graphs}
\label{sec::h_minor_free}

In this section, we prove \cref{minor_free_main_theorem}. We begin by reviewing the main parts of the Graph Minor Structure Theorem that are used in our proof.

\subsection{The Graph Minor Structure Theorem}

Let $G_0$ be a graph embedded on a surface $\Sigma$. A closed disk $D \subseteq \Sigma$ is \defin{$G_0$-clean} if the interior of $D$ is disjoint from $G_0$, and the boundary of $D$ only intersects $G_0$ in vertices of $V(G_0)$. Let $x_1, \dots, x_s$ be the vertices of $G_0$ on the boundary of $D$ in the order they occur around the boundary of $D$. A \defin{$D$-vortex} consists of a graph $H$ and a path-decomposition $(B_{x_1}, \ldots, B_{x_s})$ of $H$ such that $x_i \in B_{x_i}$ for each $i \in \{1, \ldots, s\}$ and $V(G_0 \cap H) = \{x_1, \dots, x_s\}$.

For integers $a, g, r \geq 0$ and $w \geq 1$, an \defin{$(a, g, r, w)$-near-embedding} of a graph $G$ is a tuple $\calE := (A, G_0, G_1, \dots, G_r)$ such that:
\begin{enumerate}[nosep,nolistsep]
	\item  $A \subseteq V(G)$ with $|A| \leq a$.
	\item $G_0, G_1, \dots, G_r$ are subgraphs of $G$ such that $G \setminus A = G_0 \cup \dots \cup G_r$.
	\item $G_1, \dots, G_r$ are pairwise vertex-disjoint.
	\item $G_0$ is embedded on a surface $\Sigma$ of Euler genus at most $g$,
	\item There are pairwise disjoint $G_0$-clean disks $D_1, \dots, D_r$ in $\Sigma$, and
	\item $G_i$ is a $D_i$-vortex whose path-decomposition has width at most $w$, for each $i \in \{1,\dots, r\}$.
\end{enumerate}
The vertices of $A$ are called \defin{apex vertices}. The apex vertices can be adjacent to any vertex of $G$. We say a graph is \defin{$\ell$-near-embeddable}  if it has an $(a, g, r, w)$-near-embedding for some $a, g, r, w \leq \ell$.

The following result, variants of which appear in \citet{grohe2003local,dvorak.kral.ea:testing}, states that any large treewidth in an $\ell$-near-embedded graph is due to the surface-embedded part $G_0$.

\begin{lemma}[\cite{grohe2003local,dvorak.kral.ea:testing}]\label{tw_from_near_embedding}
	There exists a constant $c$ such that, for every $\ell\ge 1$ and every graph $G$ that has an $\ell$-near-embedding $(A,G_0,G_1,\ldots,G_r)$,  we have $\tw(G)\le c\ell(\tw(G_0)+1)$.
\end{lemma}

Lastly, we make use of the following variant of the Graph Minor Structure Theorem \cite{ROBERTSON200343} due to \citet{DIESTEL20121189}.
\begin{theorem}[\citet{DIESTEL20121189}]\label{thm::diestel_kawa_muller_woll}
	For every integer $t \geq 1$, there exists an integer $\ell \geq 1$ such that every $K_t$-minor-free graph $G$ has a rooted tree-decomposition $(B_x : x \in V(T))$ such that
	\begin{enumerate}[nosep,nolistsep,label=\rm(\roman*)]
		\item for each $x \in V(T)$, the torso $G\langle B_x\rangle$ is equipped with an $\ell$-near-embedding $(A_x,G^x_0,G^x_1,\ldots,G^x_{r_x})$ in which $G^x_0$ is a $2$-cell embedded graph, and
		\item for each $x\in V(T)$ and each child $y$ of $x$, 
        \begin{enumerate}[label=\rm(\alph*)]
            \item $(B_x\cap B_y)\setminus A_x$ is contained in a bag of a vortex of $G\langle B_x\rangle$; or
            \item $(B_x\cap B_y)\setminus A_x$ consists of at most three vertices on a single face of $G^x_0$.
        \end{enumerate}
	\end{enumerate}
\end{theorem}

We make use of one easy helper lemma:

\begin{lemma}\label{treewidth_of_closure}
    Let $G$ be an $n$-vertex graph, let $\mathcal{T}:=(B_x:x\in V(T))$ be a tree-decomposition of $G$ in which each torso is an $\ell$-almost-embedded graph, and let $G^\star:=\bigcup_{x\in V(T)}\torso{G}{B_x}$.  Then the treewidth of $G^\star$ is in $O_\ell(\sqrt{n})$
\end{lemma}

\begin{proof}
    Let $(A_x,G^x_0,G^x_1,\ldots,G^x_{r_x})$ be an $\ell$-almost-embedding of $\torso{G}{B_x}$.  By \cref{excluded_grid_genus}, $\tw(G^x_0)\in O_{\ell}(\sqrt{|V(G^x_0)|})\subseteq O_\ell(\sqrt{n})$ for each $x\in V(T)$.  By \cref{tw_from_near_embedding}, $\tw(\torso{G}{B_x})\in O_\ell((\tw(G^x_0)+1))\in O_\ell(\sqrt{n})$. By \cref{lem:treewidth_bound_tree_decomp}, $\tw(G^\star)\le\max_{x\in V(T)}\{\tw(\torso{G}{B_x})\in O_\ell(\sqrt{n})$.   
\end{proof}

\subsection{Proof of \texorpdfstring{\cref{minor_free_main_theorem}}{Theorem 3}}\label{sec::minor_free_proof}

The crux of the proof of \cref{minor_free_main_theorem} is the following lemma, which is the equivalent of \cref{three_step_genus} for $H$-minor-free graphs.

\begin{lemma}\label{h_minor_free_separator}
  For every graph $H$, there exists $c_H\ge 1$ such that for all $q,n\in\N$ with $n/q^2\ge\sqrt{c_H+\log_{3/2} n}$, every $n$-vertex $H$-minor-free graph $G$ has a $q$-separator $S$ of size $O_H(n/\sqrt{q})$ and a set $S'\subseteq S$ of size $O_H(q)$ such that $\tw(G[S]-S')\in O_H(n/q^2)$.
\end{lemma}

\begin{proof}
	Let $G$, $q$, and $n$ be as given in the statement of the lemma. We can assume that $q>q_0$ for some constant $q_0$ depending only on $H$.  Otherwise, taking $S=V(G)$ and $S'=\emptyset$ works because $S$ is trivially a $q$-separator of $G$, $|S|=n\in O_H(n/\sqrt{q})$, $\tw(G[S]-S')=\tw(G)\in O_H(\sqrt{n})\subseteq O_H(n/q^2)$, and $|S'|=0\in O_H(q)$.  We now assume that $q>q_0$ for an appropriate constant $q_0$ depending only on $H$.

	Let $\calT := (B_x : x\in V(T))$ be the rooted tree-decomposition and $\ell:=\ell(H)$ be the integer obtained by applying \cref{thm::diestel_kawa_muller_woll} to $G$.  Thus, each torso of $\calT$ is equipped with an $\ell$-near-embedding.
	Let $G^\star:=\bigcup_{x\in V(T)} \torso{G}{B_x}$.
	This proof is by induction on $|V(G^\star)|$ and will establish a stronger statement:  We will show that there exist constants $c_2,c_3,c_4>0$ (depending only on $H$) such that, for every $q\ge 1$ and every $t\ge \sqrt{c_H+\log_{3/2} n}$, there exists $S\subseteq V(G^\star)$ and $S'\subseteq S$ such that
	\begin{enumerate}[nosep,nolistsep,label=(P\arabic*)]
		\item\label{req:s_qsep} $S$ is a $q$-separator of $G^\star$;
		\item\label{req:s_size} $|S|\le c_2 |V(G^\star)|/\sqrt{q}$;
		\item\label{req:st_tw} $\tw(G^\star[S]-S')\le c_3 t$; and
		\item\label{req:sprime_size} $|S'|\le c_4 |V(G^\star)|/(qt)$.
	\end{enumerate}
	To establish the lemma, we can apply the above with $t:=n/q^2$. By \ref{req:s_qsep}, $S$ is a $q$-separator of $G^\star\supseteq G$. By \ref{req:s_size}, $|S|\le c_2 n/\sqrt{q}$, so $|S|\in O_H(n/\sqrt{q})$.  By \ref{req:st_tw} we have $\tw(G^\star[S]-S')\le c_3 t$, so $\tw(G^\star[S]-S')\in O_H(n/q^2)$.  Finally, by \ref{req:sprime_size}, $|S'|\le c_4 n/(qt) = c_4 q$, so $|S'|\in O_H(q)$.

	If $|V(G^\star)|\le q$, then $S:=S':=\emptyset$ also satisfy our requirements. Now assume $|V(G^\star)|>q$.  
	
	For each $x\in V(T)$, let $G^\star_x:=\bigcup_{y\in V(T_x)}\torso{G}{B_y}=G^\star[\bigcup_{y\in V(T_x)}B_y]$.
	Let $x$ be a node of $T$ such that $|V(G^\star_x)|> q$ and $|V(G^\star_y)|\le q$ for each child $y$ of $x$.  (Such a node exists, because the root $x_0$ of $T$ satisfies  $|V(G^\star_{x_0})|=n> q^2$.)  Let $n_x:=|V(G^\star_x)|$.

	Let $\overline{T}_x:=T-V(T_x)$ and let $\overline{G}^\star_x:=G^\star[\bigcup_{y\in V(\overline{T}_x)}B_y]$ and let $\overline{n}_x:=|V(\overline{G}^\star_x)|$.  Then $(B_x:x\in V(\overline{T}_x))$ is a tree-decomposition of $\overline{G}^\star_x$ satisfying the conditions of \cref{thm::diestel_kawa_muller_woll}.  Furthermore, $|V(\overline{G}^\star_x)|\le |V(G^\star)|-n_x+|\partial_\calT(x)|\le |V(G^\star)|-q+|\partial_\calT(x)|\le|V(G^\star)|-q_0+|\partial_\calT(x)|<|V(G^\star)|$, for a sufficiently large $q_0>|\partial_\calT(x)|$ that depends only on $H$.  Thus, we may apply induction on $\overline{G}^\star_x$ with tree-decomposition $(B_x:x\in V(\overline{T}_x))$.  The result of this induction is a $q$-separator $\overline{S}_{x}$ of $\overline{G}^\star_x$ of size at most $c_2\overline{n}_x/\sqrt{q}$ and $\overline{S}_x'\subseteq \overline{S}_x$ of size at most $c_4\overline{n}_x/(qt)$ such that $\tw(\overline{G}^\star_x[\overline{S}_x]-\overline{S}_x')\le c_3 t$.

	With most of the work (seemingly) done by induction, we can now focus on $G^\star_x$. (We cannot apply induction on $G^\star_x$, because $x$ may be the root of $T$, in which case $T_x=T$ and $G^\star_x=G^\star$.)  Let $(A_x,X_0,X_1,\ldots,X_r)$ be the $\ell$-near-embedding of $\torso{G}{B_x}$.	For each $i\in\{1,\ldots,r\}$, let $f_i$ be the face of $X_0$ that contains the $X_0$-clean-disc $D_i$ that defines the vortex $X_i$ and let $(C^{(i)}_{v}:v\in D_i\cap V(X_0))$ be the path-decomposition of the vortex $X_i$.

	Our goal is to weight the faces of $X_0$ so that we can use the machinery developed in \cref{graph_surfaces}.	For each $i\in\{1,\ldots,r\}$ define $\mathdefin{d(f_i)}:=\sum_{v\in V(f_i)} 1/\deg_{X_0}(v) + |V(X_i)\setminus V(X_0)|$. For each face $f\in F(X_0)\setminus\{f_1,\ldots,f_r\}$, let $\mathdefin{d(f)}:=\sum_{v\in V(f)} 1/\deg_{X_0}(v)$.  Since $X_0$ is $2$-cell embedded, each vertex $v$ of $X_0$ appears on $\deg_{X_0}(v)$ faces of $X_0$. Therefore $\sum_{f\in F(X_0)} d(f)=|V(X_0)| + \sum_{i=1}^r |V(X_i)\setminus V(X_0)|=|B_x\setminus A_x|$. 

	The weights $d(f)$, for $f\in F(X_0)$ account for the vertices $B_x\setminus A_x$.  What remains is to account for the vertices of $G^\star_x-B_x$.
    To do this, we will associate each child $y$ of $x$ with a face \defin{$f_y$} of $X_0$ and increase the weight of $f_y$ by $|V(G^\star_y-B_x)|$. By \cref{thm::diestel_kawa_muller_woll}, for each child $y$ of $x$ one of the following applies:
    \begin{enumerate}[nosep,nolistsep,label=(\alph*)]
        \item $B_x\cap B_y\setminus A_x$ is contained in some bag of the vortex $X_i$ for some $i\in\{1,\ldots,r\}$.  In this case, we define $\mathdefin{f_y}:=f_i$.
        \item $B_x\cap B_y\setminus A_x$ is a set of at most $3$ vertices of $X_0$ on a single face $f$ of $X_0$.  In this case, we define $\mathdefin{f_y}:=f$
    \end{enumerate}
    For each face $f\in F(X_0)$, define
	\[
		\mathdefin{\delta(f)}:=\sum_{y\in N_{T_x}(x):f_y=f} |V(G^\star_y-B_x)| \enspace .
	\]
	For each face $f$ of $X_0$, define
	\[
		\mathdefin{w(f)}:=d(f)+\delta(f) \enspace .
	\]
	Then
	\[
		\sum_{f\in F(X_0)}w(f) = |B_x\setminus A_x|+\sum_{y\in N_{T_x}(x)}|V(G^\star_y-B_x)| = |V(G^\star_x-A_x)| \enspace .
	\]

	By \cref{treewidth_of_closure,q_separator_hereditary}, there exists a $q$-separator \defin{$Q_x$} of $G^-_x:=G^\star_x-(A_x\cup\partial_\calT(x))$ of size at most $c_5|V(G^\star_x)|/\sqrt{q}$.  Let $B^-_x:=B_x\setminus(A_x\cup\partial_\calT(x))$.  The separator $Q_x$ may contain vertices not in $B^-_x$. Our next step is to replace these with vertices in $B^-_x$ so that we can restrict our attention to the torso $\torso{G}{B_x}$ (and, ultimately, to the embedded part $X_0$ of $\torso{G}{B_x}$). 
    
    For each $v\in Q_x$, define the \defin{representative} \defin{$R_x(v)$} of $v$ as follows:
	\begin{itemize}[nosep,nolistsep]
		\item If $v\in B^-_x$, then $\mathdefin{R_x(v)}:=\{v\}$.

		\item Otherwise, $v$ is a vertex of $G^\star_y-B_x$ for some child $y$ of $x$.  Then $\mathdefin{R_x(v)}:=B^-_x\cap B_y$.
	\end{itemize}
	Observe that, for each $v\in Q_x$, $R_x(v)\subseteq B^-_x$ and $|R_x(v)|\le\ell$.
	Let $\mathdefin{R_x}:=\bigcup_{v\in Q_x} R_x(v)$.	Then $|R_x|\le\ell|Q_x|\le c_5\ell |V(G^\star_x)|/\sqrt{q}$ and $Q_x\cap B^-_x\subseteq R_x\subseteq B^-_x$.

  By definition, $Q_x\cup A_x\cup\partial_\calT(x)$ is a $q$-separator of $G^\star_x$.  We now argue that $\mathdefin{S_x}:=R_x\cup A_x\cup\partial_\calT(x)$ is a $q$-separator of $G^\star_x$.  Let $C$ be a component of $G^\star_x-S_x$, so $V(C)\cap S_x=\emptyset$.  If $V(C)\subseteq B^-_x$, then $V(C)\cap Q_x=\emptyset$, so $C$ is a connected subgraph of $G^-_x-Q_x$.  In this case $|V(C)|\le q$ since $Q_x$ is a $q$-separator of $G^-_x$.  If $C$ is a subgraph of $G^\star_y-B_x$ for some child $y$ of $x$, then $|V(C)|\le |V(G^\star_y)|\le q$ since $x$ has no child with $|V(G^\star_y)|>q$.  The only remaining possibility is that $C$ contains a vertex in $B_x\setminus S_x$ and $C$ contains a vertex in $G^\star_y-B_x$ for some child $y$ of $x$. In this case $C$ contains a vertex of $G^\star_y-B_x$, so $R_x$ contains $B^-_x\cap B_y$ and $S_x$ contains $B_x\cap B_y$.  But this is not possible, since $B_x\cap B_y$ separates $G^\star_y-B_x$ from $B_x\setminus S_x$. Therefore, every component $C$ of $G^\star_x-S_x$ has at most $q$ vertices, so $S_x$ is a $q$-separator of $G^\star_x$.

	The tools in \cref{graph_surfaces}, in particular \cref{surface_tw_by_face_weights}, require inclusion-minimal separators.  We cannot eliminate vertices in $A_x$ from $S_x$ because they may be adjacent to any vertex in $B_x$.  We cannot eliminate vertices in $\partial_\calT(x)$ from $S_x$ because we rely on them to separates $G^\star_x$ from $\overline{G}^\star_x$.  Instead, we make $R_x\cap V(X_0)$ inclusion-minimal by repeatedly removing any vertex $v\in R_x\cap V(X_0)$ such that $(R_x\cup A_x\cup\partial_\calT(x))\setminus\{v\}$ is a $q$-separator of $G^\star_x$.  Redefine $\mathdefin{S_x}:=R_x\cup A_x\cup\partial_\calT(x)$ based on the (now inclusion-minimal) set $R_x$.

	The next step is to bound the treewidth of $X_0[R_x]$ using \cref{surface_tw_by_face_weights} by weighting the faces of $X_0[R_x]$ using the weight function $w$. for each $f\in F(X_0[R_x])$, let $w(f):=\sum_{f'\in F(X_0):f'\subseteq f}w(f')$.  Since each face of $X_0$ is contained in exactly one face of $X_0[R_x]$, we have $\sum_{f\in F(X_0[R_x])}w(f)=\sum_{f\in F(X_0)}w(f)=|V(G^\star_x-A_x)|$.

	We now show that, for each $v\in R_x\cap V(X_0)$, the faces of $X_0[R_x]$ incident to $v$ have total weight at least $q$; that is, $\sum_{f\in F(X_0[R_x],v)}w(f)>q$. Let $v\in R_x\cap V(X_0)$ and let $C$ be the component of $G^\star_x-(S_x\setminus\{v\})$ that contains $v$. Then $|V(C)|>q$, otherwise $v$ would have been removed from $R_x$ while making $R_x\cap V(X_0)$ inclusion-minimal.  Fix a vertex $u\neq v$ in $C$.  We want to show that $u$ contributes $1$ to $w(f)$ for some $f\in F(X_0[R_x],v)$.  This ensures that each vertex in $C-v$ contributes $1$ to $\sum_{f\in F(X_0[R_x],v)} w(f)$, so that $\sum_{f\in F(X_0[R_x],v)} w(f) \ge |V(C)-1|\ge q$.
    
	Let $P$ be a shortest path from $u$ to $v$ in $C$. Let $u'$ be the first vertex of $P$ that is in $B_x$. (The following arguments consider the possibility that $u'=v$.)  Thus $u'\in B^-_x$.  We claim that the subpath $P'$ of $P$ from $u'$ to $v$ is entirely contained in $\torso{G}{B_x}-(S_x\setminus\{v\})$.  Indeed, the only other possibility is that $P'$ contains a vertex $w$ in $G^\star_y-B_x$ for some child $y$ of $x$.  Since $B_x\cap B_y$ separates $w$ from $\{v,u'\}$, there is a subpath of $P'$ that contains $w$ and that contains two vertices $w', w''\in B^-_x\cap B_y$. However, $B^-_x\cap B_y$ is a clique in $G^-_x$, so this subpath of $P'$ could be replaced by the edge $w'w''$, contradicting the fact that $P$ is a shortest path.  Thus, $P'$ is a path from $u'$ to $v$ in $\torso{G}{B_x}-(S_x\setminus\{v\})$. 

	Let $x_0,\ldots,x_s:=P'$, so $x_0=u'$ and $x_s=v$.  With a small abuse of terminology, let us say, for each $j\in\{1,\ldots,r\}$ and each $u\in V(X_j)\setminus V(X_0)$ that $u$ \defin{is in the interior} of $f_j$.  Taking this abuse one step further, if $u$ is in the interior of $f_j$ and $f_j\subseteq g$ for some face $g$ of $X_0[R_x]$, we also say that $u$ is in the interior of $g$.  We will now show that $x_0,\ldots,x_{s-1}$ are all in the interior of the same face $f_u$ of $X_0[R_x]$.  Furthermore, if $u'\neq v$, then $f_u$ is incident to $v=x_s$. For each $i\in\{1,\ldots,s\}$, let $g_i$ be the face of $X_0[R_x]$ that contains $x_{i-1}$ in its interior.  This is well-defined because $x_{i-1}\not\in R_x$, by the definition of $P$ and $C$.  Then the edge $x_{i-1}x_i$ is either an edge of $X_0$ or is an edge of $X_j$ for some $j\in\{1,\ldots,r\}$. In either case, $x_i$ is either in the interior of $g_i$ (if $i <s$, because $x_i\not\in R_x$) or $x_i$ is on the boundary of $g_i$ (if $i=s$).  Thus, $g_1=g_2=\cdots=g_s$.  Furthermore, if $u'\neq v$ then $f_u:=g_1=\cdots=g_s$ is a face of $X_0[R_x]$ incident to $v=x_s$.  (We will deal with the case $u'=v$ separately, below).  Now, there are two possibilities for the vertex $u$ that we are trying to account for:
	\begin{enumerate}[nosep,nolistsep]
			\item If $u'=u$ then $u'=u\neq v$ by the definition of $u$, so $x_0=u\not\in R_x$ is in the interior of $f_u$. There are two cases:
			\begin{enumerate} 
				\item $u\in V(X_0)$: in this case, $u$ contributes $1/\deg_{X_0}(u)$ to each incident face of $X_0$.  Since $u\not\in R_x$, each face of $X_0$ incident to $u$ is contained in $f_u$, so $u$ contributes $1$ to $w(f_u)$.
				\item $u\in V(X_j)\setminus V(X_0)$ for some $j\in\{1,\ldots,r\}$: in this case, $u$ is in the interior of $f_j$, so $u$ contributes $1$ to $\delta(f_j)$, and $f_j\subseteq f_u$.
			\end{enumerate}
			\item If $u\neq u'$ then $u$ is a vertex of $G^\star_y-B_x$ for some child $y$ of $x$.
			\begin{enumerate}[label=\rm(\alph*)]
					\item If $B_x\cap B_y\setminus A_x$ is contained in a bag of $X_j$ for some $j\in\{1,\ldots,r\}$, then $f_y=f_j$.  Then $u$ contributes $1$ to $\delta(f_j)$. The vertex  $u'$ is in the interior of $f_j$ and $f_j\subseteq f_u$ (if $u'\neq v$) or $f_j$ is incident to $v$ (if $u'=v$). In either case, $u$ contributes $1$ to $\delta(f)$ (and therefore to $w(f)$) for some $f\in F(X_0[R_x],v)$.
										
					\item Otherwise, $B_x\cap B_y\setminus A_x$ is contained in a single face $f_y$ of $X_0$ and $f_y$ is incident to $u'$.  Then $f_y\subseteq f_u$ (if $u'\neq v$) or $f_y$ is incident to $v$ (if $u'=v$). In either case, $u$ contributes $1$ to $\delta(f)$ (and therefore to $w(f)$) for some $f\in F(X_0[R_x],v)$. 
			\end{enumerate}
	\end{enumerate}
	Summarizing, every vertex $u\in V(C-\{v\})$ is counted in $w(f_u)$ for some face $f_u\in F(X_0[R_x],v)$, so $\sum_{f\in F(X_0[R_x],v)}w(f)\ge |V(C)-1|\ge q$.  Therefore, by \cref{surface_tw_by_face_weights}, $X_0[R_x]$ has treewidth at most $c_6\sqrt{n_x/q}$.  The graph $G^\star[S_x]$ inherits an $\ell$-near-embedding, $(A_x,X_0[S_x],X_1[S_x],\ldots,X_r[S_x])$ from $\torso{G}{B_x}$. By \cref{tw_from_near_embedding}, $\tw(G^\star_x[S_x])\le c_7(\tw(X_0[S_x])+1)\le c_7(\tw(X_0[R_x])+1)+|S_x\setminus R_x|\le c'_7(\tw(X_0[R_x])+1)$.	
	
	If $n_x\le t^2q$, then we define  $S'_x=\emptyset$.  In this case,  $\tw(G^\star[S_x]-S'_x)=\tw(G^\star_x[S_x])\le c_7(c_6\sqrt{n_x/q}+1)+ |S_x\setminus R_x| \le c'_7(c_6\sqrt{n_x/q}+1)\le c'_7(c_6+1) t \le c_3 t$ (by choosing $c_3 \ge c'_7(c_6+1)$).

	If $n_x > t^2 q$ then there is more work to do.  
	By \cref{planarizer_tw}, $X_0[R_x]$ has a planarizer $P_x$ of size at most $c_8\sqrt{n_x/q}$ and $X_0[R_x]-P_x$ is a plane graph with at most $2\ell$ faces that are not faces of $X_0[R_x]$. Let $Y:=F(X_0[R_x]-P_x)\setminus F(X_0[R_x])$, so $|Y|\le 2\ell$. Recall that $t\geq \sqrt{c_H+\log_{3/2} n}\ge\sqrt{1+2\ell+\log_{3/2} n}$ for $c_H\ge 1+2\ell$. Thus, $X_0[R_x]-P_x$ is a plane graph, $w(f)$ is defined for all $f\in F(X_0[R_x]-P_x)\setminus Y$, $|Y|\le 2\ell$, $t\ge\sqrt{1+|Y|+\log_{3/2} n}$ and, for each $v\in V(X_0[R_x]-P_x)$, either $F(X_0[R_x-P_x],v)\cap Y\neq\emptyset$ or $\sum_{f\in F(X_0[R_x]-P_x,v)}w(f)\ge q$.  Therefore, by \cref{lem::reduce_tw}, applied to the plane graph $X_0[R_x]-P_x$, with face weighting $w:F(X_0[R_x]-P_x)\setminus Y\to\N^+$, there exists $S''_x\subseteq R_x\setminus P_x$ of size at most $c_9 n_x/(qt)$ such that $\tw(X_0[R_x]-(P_x\cup S''_x))\le c_9 t$. Let $S'_x:=P_x\cup S''_x$. Again, \cref{tw_from_near_embedding} implies that $\tw(G^\star_x[S_x]-S'_x)\le \tw(G^\star_x[R_x]-S'_x)+|S_x\setminus R_x|\le c_7(\tw(X_0[R_x]-S'_x)+1)+|S_x\setminus R_x| \le c'_7(c_9 t+1) \le c_{10} t \le c_3 t$ (by letting $c_{10} = c'_7(c_9+1)$ and choosing $c_3 \ge c_{10}$).

  Finally, let $S:=S_x\cup \overline{S}_x$ and $S':=S'_x\cup \overline{S}'_x$.  We now verify that $S$ and $S'$ satisfy the required properties:
	\begin{enumerate}[nosep,nolistsep,label=(P\arabic*)]
		\item Since $\partial_{\calT}(x)\subseteq S_x\subseteq S$, each component $C$ of $G^\star-S$ is contained in $G^\star_x-S_x$ or is contained in $\overline{G}^\star_x-\overline{S}_x$.  In the former case, $C$ has size at most $q$ since $S_x$ is a $q$-separator of $G^\star_x$.  In the latter case, $C$ has size at most $q$ since $\overline{S}_x$ is a $q$-separator of $\overline{G}^\star_x$.
		\item $|S|\le |S_x| + |\overline{S}_x| \le \ell c_5 n_x/\sqrt{q} + 3\ell + c_2\overline{n}_x/\sqrt{q} \le c_2(n_x+\overline{n}_x-2\ell)/\sqrt{q} \le c_2 n/\sqrt{q}$ (by choosing $c_2 \ge \ell c_5 + 5\ell$).
		\item By \cref{lem:treewidth_bound_tree_decomp}, we have that $\tw(G^\star[S]-S')\le \max\{ \tw(G^\star_x[S_x]-S'_x), \tw(\overline{G}^\star_x[\overline{S}_x]-\overline{S}'_x) \}\le \max\{c_3 t, c_3 t\} = c_3 t$.  
		\item To bound $|S'|$, there are two cases to consider:
		\begin{itemize}[nosep,nolistsep]
			\item If $n_x\le t^2 q$ then $S'_x=\emptyset$, so $|S'|=|\overline{S}'_x|\le c_4\overline{n}_x/(qt)\le c_4(n-n_x+2\ell)/(qt)\le c_4 n/(qt)$.
			\item If $n_x > t^2 q$ then $|S'|=|\overline{S}'_x|+|S'_x|\le c_4\overline{n}_x/(qt)+(c_8+c_9)n_x/(qt)\le c_4 n/(qt)$ (by choosing $c_4$ sufficiently large). \qedhere
		\end{itemize}
	\end{enumerate}
\end{proof}

\begin{proof}[Proof of \cref{thm::3-colouring_H_free_minor}]
	Let $G$ be an $n$-vertex $H$-minor-free graph.
    Let $q:=n^{4/9}$ so that $t:=n/q^2=n^{1/9}$. If $n^{1/9}<\sqrt{c_H+\log_{3/2} n}$ then $n^{2/9}-\log_{3/2} n < c_H$ so $n\in O_H(1)$.  In this case we can colour all vertices of $G$ red and the resulting colouring has clustering $O_H(1)$. Now assume that $n^{1/9}\ge\sqrt{c_H+\log_{3/2} n}$.
    
    By \cref{h_minor_free_separator} there exists an $q$-separator $S$ of $G$ of size $O_H(n/q^{1/2}) = O_H(n^{1 - 2/9}) = O_H(n^{7/9})$ and $S'\subseteq S$ of size $O_H(q) = O_H(n^{4/9})$ such that $\tw(G[S]-S') \in O_H(n/q^{2}) = O_H(n^{1-8/9}) = O_H(n^{1/9})$.  By \cref{lem:tw_separator_bound}, there exists a $q$-separator $S''\subseteq S\setminus S'$ of $G[S]-S'$ of size $O_H((|S\setminus S'|/q)\cdot\tw(G[S]-S')) = O_H((n^{7/9}/n^{4/9})\cdot n^{1/9}) = O_H(n^{4/9})$. Colour each component of $G-S$ red. Colour each component of $G[S]-(S'\cup S'')$ blue. Colour the vertices of $S'\cup S''$ green.

    Each red component has size at most $n^{4/9}$ because $S$ is a $q$-separator of $G$.  Each green component has size at most $n^{4/9}$ because $S''$ is a $q$-separator of $G[S]-S'$, so $S''\cup S'$ is a $q$-separator of $G[S]$.  Each blue component has size $O(n^{4/9})$ because $|S'\cup S''|\in O(n^{4/9})$.
\end{proof}

\section{Conclusion}

This work leaves open the question of determining the optimal clustering achievable for $3$-colouring graphs from minor-closed graph families.  The current best upper bound is $O(n^{4/9})$ from the current paper. The best known lower bound is $\Omega(n^{4/10})$ from \cite{LINIAL_MATOUŠEK_SHEFFET_TARDOS_2008}, which applies already to $K_6$-minor-free graphs.

\section*{Acknowledgements}

This question was posed at the Thirteenth Annual Workshop on Geometry and Graphs, Jan 30–Feb 6, 2026 at the Bellairs Research Institute of McGill University in Holetown, Barbados. The authors are grateful to the organizers and workshop participants for creating a stimulating research environment.

\bibliographystyle{plainurlnat}
\bibliography{references}
\end{document}